\documentclass[11pt]{amsart}

\usepackage[a4paper,vmargin=4cm]{geometry}
\usepackage{amsfonts,amssymb,amscd,amstext}
\usepackage{graphicx}
\usepackage{color}
\usepackage[dvips]{epsfig}

\usepackage[utf8]{inputenc}
\usepackage{hyperref}

\usepackage{verbatim}
\usepackage{fancyhdr}
\input xy
\xyoption{all}

\usepackage{times}
\usepackage{enumerate}
\usepackage{titlesec}
\usepackage{mathrsfs}

\numberwithin{equation}{section}
\numberwithin{figure}{section}

\titleformat{\section}
{\filcenter\bfseries\large} {\thesection{.}}{0.2cm}{}
\titleformat{\subsection}[runin]
{\bfseries} {\thesubsection{.}}{0.15cm}{}[.]
\titleformat{\subsubsection}[runin]
{\em}{\thesubsubsection{.}}{0.15cm}{}[.]

\usepackage[up,bf]{caption}

\newtheorem{theorem}{Theorem}[section]
\newtheorem{proposition}[theorem]{Proposition}

\newtheorem{lemma}[theorem]{Lemma}
\newtheorem{corollary}[theorem]{Corollary}

\theoremstyle{definition}
\newtheorem{definition}[theorem]{Definition}
\newtheorem{remark}[theorem]{Remark}

\newtheorem{problem}[theorem]{Problem}
\newtheorem{example}[theorem]{Example}

\newcommand\Ocal{\mathcal{O}}

\newcommand\Ascr{\mathscr{A}}

\newcommand\Cscr{\mathscr{C}}

\newcommand\Oscr{\mathscr{O}}

\newcommand\B{\mathbb{B}}
\newcommand\C{\mathbb{C}}
\newcommand\D{\overline{\mathbb D}}

\renewcommand\D{\mathbb D}

\newcommand\N{\mathbb{N}}

\newcommand\R{\mathbb{R}}

\newcommand\igot{\mathfrak{i}}

\renewcommand\igot{\mathfrak{i}}

\renewcommand\imath{\igot}

\newcommand\hra{\hookrightarrow}

\newcommand\di{\partial}

\newcommand\dist{\mathrm{dist}}

\newcommand\Conv{\mathrm{Conv}}

\def\dist{\mathrm{dist}}

\def\Ell1{\mathrm{Ell_1}}
\def\CEll1{\mathrm{CEll_1}}

\def\cB{\overline{\B}} 

\begin{document}

\fancyhead[LO]{Proper holomorphic maps in Euclidean spaces avoiding unbounded convex sets}
\fancyhead[RE]{B.\ Drinovec Drnov\v sek and F.\ Forstneri\v c} 
\fancyhead[RO,LE]{\thepage}

\thispagestyle{empty}


\begin{center}
{\bf \LARGE Proper holomorphic maps in Euclidean spaces \\ avoiding unbounded convex sets}

\vspace*{0.5cm}

{\large\bf  Barbara Drinovec Drnov\v sek and Franc Forstneri{\v c}} 
\end{center}

\vspace*{0.5cm}

{\small
\noindent {\bf Abstract} 
%
We show that if $E$ is a closed convex set in $\C^n$ $(n>1)$ contained in a closed halfspace 
$H$ such that $E\cap bH$ is nonempty and bounded, then the concave domain 
$\Omega = \C^n\setminus E$ contains images of proper holomorphic maps
$f:X\to \C^n$ from any Stein manifold $X$ of dimension $<n$, with approximation
of a given map on closed compact subsets of $X$. 
If in addition $2\dim X+1\le n$ then $f$ can be chosen an embedding, and if $2\dim X=n$ 
then it can be chosen an immersion. Under a stronger condition on $E$ we also obtain
the interpolation property for such maps on closed complex subvarieties.

\hspace*{0.1cm}
}

\noindent{\bf Keywords}\hspace*{0.1cm} 
Stein manifold, holomorphic embedding, Oka manifold, minimal surface, convexity

\vspace*{0.1cm}

\noindent{\bf MSC (2010):}\hspace*{0.1cm}  32H02, 32Q56; 52A20, 53A10
%
%
%
%
%

\noindent {\bf Date: \rm 3 January 2023}


%
%

\bigskip\medskip

\centerline{\em In memoriam Nessim Sibony}

\bigskip

\section{Introduction}\label{sec:intro}  
Let $X$ be a Stein manifold. Denote by $\Oscr(X,\C^n)$ the Frechet space of holomorphic maps 
$X\to\C^n$ endowed with the  compact-open topology and write $\Oscr(X,\C)=\Oscr(X)$. 
A theorem of Remmert \cite{Remmert1956} (1956), Narasimhan \cite{Narasimhan1960AJM} (1960), and Bishop \cite{Bishop1961AJM} (1961)
states that almost proper maps are residual in $\Oscr(X,\C^n)$ if $\dim X=n$, 
proper maps are dense if $\dim X<n$, proper immersions are dense if $2\dim X \le n$, and 
proper embeddings are dense if $2\dim X<n$. A proof is also given in the 
monograph \cite{GunningRossi2009} by Gunning and Rossi. 

It is natural to ask how much space proper maps need. We pose the following question.

%
%
\begin{problem}\label{pr:problem1}
For which domains $\Omega\subset \C^n$ are proper holomorphic maps 
(immersions, embeddings) $X\to \C^n$ as above,
with images contained in $\Omega$, dense in $\Oscr(X,\Omega)$?
\end{problem}

It is evident that $\Omega$ cannot be contained in a halfspace of $\C^n$ since 
every holomorphic map from $\C$ to a halfspace lies in a complex
hyperplane. In this paper we give an affirmative answer for concave domains whose complement $E=\C^n\setminus \Omega$ 
satisfies the following condition.

%
%
\begin{definition} \label{def:BCEH}
A closed convex set $E$ in a real or complex Euclidean space $V$ has 
{\em bounded convex exhaustion hulls} (BCEH) if for every compact convex set $K$ in $V$
\begin{equation}\label{eq:h}
	\text{the set}\ \ h(E,K) = \Conv(E\cup K)\setminus E \ \ \text{is bounded.}
\end{equation}
\end{definition}

Here, $\Conv$ denotes the convex hull. The following is our first main result.

%
%
\begin{theorem}\label{th:main}
Let $E$ be an unbounded closed convex set in $\C^n$ $(n>1)$ with  
bounded convex exhaustion hulls. 
Given a Stein manifold $X$ with $\dim X<n$, a compact $\Ocal(X)$-convex set $K$ in $X$, 
and a holomorphic map $f_0:K\to\C^n$ with $f_0(bK)\subset \Omega=\C^n\setminus E$, 
we can approximate $f_0$ uniformly on $K$ by proper holomorphic maps $f:X\to \C^n$ 
satisfying $f(X\setminus \mathring K) \subset \Omega$. The map $f$ can be chosen an 
embedding if $2\dim X < n$ and an immersion if $2\dim X \le n$. 
\end{theorem}

In this paper, a map $f:K\to \C^n$ from a compact set $K$ is said to be holomorphic 
if it is the restriction to $K$ of a holomorphic map on an open neighbourhood of $K$.

In particular, if $f_0(K)\subset \Omega$ then the theorem gives uniform approximation
of $f_0$ by proper holomorphic maps $f:X\to\C^n$ with $f(X)\subset \Omega$.
If $bE$ is of class $\Cscr^1$ and strictly convex near infinity, 
we obtain an analogue of Theorem \ref{th:main}
with additional interpolation on a closed complex subvariety $X'$ of $X$
such that $f_0:X'\to \C^n$ is proper holomorphic;  
see Theorem \ref{th:interpolation}. Without the condition on the range,
interpolation of proper holomorphic embeddings $X\hra\C^n$ 
on a closed complex subvariety was obtained by 
Acquistapace et al.\ \cite{AcquistapaceBrogliaTognoli1975} in 1975.

The analogue of the BCEH condition for unbounded closed sets 
in Stein manifolds, with the convex hull replaced by the holomorphically convex hull, 
is used in holomorphic approximation theory of Arakelyan and Carleman type; 
see the survey in \cite{FornaessForstnericWold2020}.

It is evident that a closed convex set $E\subset\R^n$ 
has BCEH if and only if there is an increasing sequence
$K_1\subset K_2\subset \cdots$ of compact convex sets exhausting $\R^n$
such that the set $h(E,K_j)$ (see \eqref{eq:h}) is bounded for every $j=1,2,\ldots$. 
In particular, BCEH is a condition at infinity which is invariant under perturbations
supported on a compact subset.   For compact convex sets $E\subset\C^n$, 
Theorem \ref{th:main} was proved in \cite{ForstnericRitter2014}; 
in this case BCEH trivially holds.

We show in Section \ref{sec:BCEH} that a closed convex set $E$ in $\R^n$ has BCEH if 
and only if $E$ is continuous in the sense of Gale and Klee \cite{GaleKlee1959}; 
see Proposition \ref{prop:BCEHcontinuous}. 
If $E$ has BCEH then $\Conv(E\cup K)$ is closed for any compact convex set 
$K\subset\R^n$ (see \cite[Theorem 1.5]{GaleKlee1959}).
If such $E$ is unbounded, which is the main case of interest, there are 
affine coordinates $(x,y)\in \R^{n-1}\times\R$ such that 
$
	E=E_\phi =\{(x,y)\in\R^n: y \ge \phi(x)\}
$
is the epigraph of a convex function $\phi:\R^{n-1}\to \R_+=[0,+\infty)$ growing at least 
linearly near infinity (see Proposition \ref{prop:epigraph}). In particular, an unbounded 
closed convex set $E\subset \C^n$ with BCEH is of the form 
\begin{equation}\label{eq:EinCn}
	E =E_\phi=\{z=(z',z_n)\in\C^n: \Im z_n \ge \phi(z',\Re z_n)\}
\end{equation}
in some affine complex coordinates $z=(z',z_n)$ on $\C^n$, with $\phi$ as above. 
(Here, $\Re$ and $\Im$ denote the real and the imaginary part.)
For a convex function $\phi$ of class $\Cscr^1$ we give a differential characterization
of the BCEH condition on its epigraph $E_\phi$; see Proposition \ref{prop:BCEH}.
The BCEH property holds if the radial derivative of $\phi$
tends to infinity; see Corollary \ref{cor:derivative}.
On the other hand, there are convex functions of linear growth
whose epigraphs have BCEH; see Example \ref{ex:lineargrowth}.
By Proposition \ref{prop:approximation}, a convex function $\phi$ 
with at least linear growth at infinity can be approximated uniformly on compacts 
by functions $\psi\le \phi$ of the same kind whose epigraphs $E_\psi$ 
have BCEH. This allows us to extend Theorem \ref{th:main} as follows; 
see Section \ref{sec:proof} for the proof. 

%
%
\begin{corollary}\label{cor:main}
The conclusion of Theorem \ref{th:main} holds for any convex epigraph $E_\phi$ 
of the form \eqref{eq:EinCn} such that $\phi\ge 0$ and the set $\{\phi=0\}$ 
is nonempty and compact.
\end{corollary}

A closed convex set $E\subset \C^n$ with BCEH does not contain any affine real line
(see Proposition \ref{prop:epigraph}), and for $n>1$ 
its complement $\Omega=\C^n\setminus E$
is an Oka domain according to Wold and the second named author; 
see \cite[Theorem 1.8]{ForstnericWold2022Oka}. This fact plays an important
role in our proof of Theorem \ref{th:main}, given in Section \ref{sec:proof}.
(The precise result from Oka theory which we shall use is stated as 
Theorem \ref{th:Oka}.)
Among closed convex epigraphs \eqref{eq:EinCn}, the class of sets 
with Oka complement is strictly bigger than the class of sets with BCEH. 
In particular, the former class contains many sets containing boundary lines, 
which is impossible for a set with BCEH.  

%
%
\begin{problem}\label{prob:sublinear}
Is there a (not necessarily convex) set $E_\phi\subset \C^n$ of the form \eqref{eq:EinCn} with 
$\phi\ge 0$ of sublinear growth for which Theorem \ref{th:main} holds?
Is there a set of this kind in $\C^2$ such that $\C^2\setminus E_\phi$ contains 
the image of a proper holomorphic disc $\D=\{z\in\C:|z|<1\}\to\C^2$? 
\end{problem}

Theorem \ref{th:main} is the first general result in the literature 
providing proper holomorphic maps $X\to \C^n$ from any Stein manifold of dimension 
$<n$ whose images avoid large convex sets in $\C^n$ close to a halfspace, 
and with approximation of a given map on a compact holomorphically convex set in $X$. 
Without the approximation condition and assuming that $\dim X\le n-2$, 
there are proper holomorphic maps of $X$ into a complex hyperplane in 
$\C^n\setminus E$. 

On the other hand, there are many known results concerning 
proper holomorphic maps in Euclidean spaces and in more general
Stein manifolds whose images avoid certain small closed 
subsets, such as compact or complete pluripolar ones, and results in which the source manifold 
is the disc $\D=\{z\in\C:|z|<1\}$. Proper holomorphic discs in $\C^2$ avoiding closed complete 
pluripolar sets of the form $E=E'\times\C$, with $E'\subset\C$, were constructed by 
Alexander \cite{Alexander1977} in 1977. The first named author showed in 
\cite{Drinovec2004MRL} (2004) that for every closed complete pluripolar set $E$ in a 
Stein manifold $Y$ with $\dim Y>1$ and point $p\in Y\setminus E$ there is a proper holomorphic 
disc $f:\D\to Y$ with $p\in f(\D)\subset Y\setminus E$. 
If $Y=\C^2$ there also exist embedded holomorphic discs with this property according to 
Borell et al.\ \cite{BorellKutzschebauchWold2008} (2008), and for $\dim Y\ge 3$ this holds 
by the general position argument. Proper holomorphic discs in $\C^2$ 
with images contained in certain concave cones were constructed by Globevnik and the 
second named author \cite{ForstnericGlobevnik2001MRL} in 2001. 
They also constructed proper holomorphic discs in $\C^2$ with images in 
$(\C\setminus \{0\})^2$, and hence proper harmonic discs $\D\to \R^2$,
disproving a conjecture by Schoen and Yau \cite[p.\ 18]{SchoenYau1997}.
(Another construction of such maps was given by Bo\v zin \cite{Bozin1999IMRN}.)
More generally, it was shown by Alarc\'on and L\'opez \cite[Corollary 1.1]{AlarconLopez2012JDG}
in 2012 that every open Riemann surface $X$ admits a proper harmonic map to $\R^2$ 
which is the projection of a conformal minimal immersion $X\to\R^3$.
The aforementioned result from \cite{ForstnericGlobevnik2001MRL} 
was used by the first named author in \cite{Drinovec2002MZ} 
(2002) to classify closed convex sets in $\C^2$ whose complement is 
filled by images of holomorphic discs which are proper in $\C^2$. 
More recently, Forstneri\v c and Ritter \cite{ForstnericRitter2014} (2014) 
proved Theorem \ref{th:main} in the case when $E\subset\C^n$ is 
a compact polynomially convex set and $2\dim X\le n$ (for immersions) or 
$2\dim X<n$ (for embeddings), and for proper holomorphic maps $X\to\C^n$ 
when $\dim X<n$ and $E$ is a compact convex set. A further development 
in this direction is the analogue of Theorem \ref{th:main} when 
$\C^n$ is replaced by a Stein manifold $Y$ with the density property
and $E\subset Y$ is a compact $\Oscr(Y)$-convex set; see 
\cite[Remark 4.5]{Forstneric2022Oka} and the references therein.
However, in all mentioned results except those in 
\cite{ForstnericGlobevnik2001MRL,Drinovec2002MZ},
the avoided sets are thin or compact.

%
%
Without insisting on approximation, the theorem of Remmert, Bishop, and Narasimhan
is not optimal with respect to the dimension of the target space.
Indeed, it was shown by Eliashberg and Gromov \cite{EliashbergGromov1992} in 1992,
with an improvement for odd dimensional Stein manifolds by Sch\"urmann \cite{Schurmann1997} 
in 1997, that a Stein manifold $X$ of dimension $m\ge 2$ embeds properly 
holomorphically in $\C^n$ with $n=\left[\frac{3m}{2}\right] +1$, and for $m\ge 1$ it immerses 
properly holomorphically in $\C^n$ with $n=\left[\frac{3m+1}{2}\right]$. 
(See also \cite[Sect.\ 9.3]{Forstneric2017E}.)   
However, the construction method in these papers, which relies 
on the Oka principle for sections of certain stratified holomorphic fibre bundles, 
does not give the density statement, and we do not know whether 
Theorem \ref{th:main} holds for maps to these lower dimensional spaces. 
It is also an open problem whether every open Riemann surface 
embeds properly holomorphically in $\C^2$; see \cite[Secs.\ 9.10-9.11]{Forstneric2017E}
and the survey \cite{Forstneric2018Korea}. 

%
%
Theorem \ref{th:main} is proved in Section \ref{sec:proof}. The proof relies on two main 
ingredients. One is the result of Wold and the second named author 
\cite[Theorem 1.8]{ForstnericWold2022Oka} 
which shows in particular that the complement $\Omega=\C^n\setminus E$ 
of a closed convex set $E$ having BCEH is an Oka domain.
The second main technique comes from the work of Dor \cite{Dor1993,Dor1995} (1993-95), 
following earlier papers by Stens\o nes \cite{Stensones1989} (1989) 
and Hakim \cite{Hakim1990} (1990). Dor constructed proper 
holomorphic immersions and embeddings of any smoothly bounded, relatively compact, 
strongly pseudoconvex domain $D$ in a Stein manifold $X$ into any pseudoconvex domain 
$\Omega$ in $\C^n$ under the dimension conditions in Theorem \ref{th:main}.
Previously, Hakim \cite{Hakim1990} constructed proper holomorphic maps to balls in 
codimension one. The main idea is to inductively lift the image of $bD$ under a 
holomorphic map $f:\bar D\to \Omega$ to a given higher superlevel set of a strongly 
plurisubharmonic exhaustion function $\rho:\Omega\to\R_+$ in a controlled way, taking care 
not to decrease the value of $\rho\circ f$ very much anywhere on $D$
during the process. When $D$ is a finite bordered Riemann surface, this can be 
achieved by using approximate solutions of a Riemann-Hilbert boundary 
value problem (see \cite{DrinovecForstneric2007DMJ}).
In higher dimensions the proof is more subtle and uses carefully controlled 
holomorphic peak functions on $\bar D$ to push a given map $f:\bar D\to \Omega$ 
locally at a point $z\in f(bD)$ in the direction of the zero set $S_z$ of the holomorphic 
(quadratic) Levi polynomial of the exhaustion function $\rho:\Omega\to\R$. 
At a noncritical point $z\in \Omega$ of $\rho$, $S_z$ is a smooth local complex 
hypersurface and the restricted function 
$\rho|_{S_z}$ increases quadratically as we move away from $z$. 
If $\rho$ is a {\em strictly convex} function,
this can be achieved by pushing the image of $f(bD)$ in the direction of 
suitably chosen affine complex hyperplanes. Dor's construction was extended by the authors  
to maps from strongly pseudoconvex domains in Stein manifolds 
to an arbitrary Stein manifold $\Omega$, and also to $q$-convex complex manifolds 
for suitable values of $q\in\N$; 
see the papers \cite{DrinovecForstneric2007DMJ,DrinovecForstneric2010AJM} 
from 2007 and 2010, respectively. In those papers we introduced the technique of 
gluing holomorphic sprays of manifold-valued maps on a strongly pseudoconvex 
Cartan pair with control up to the boundary (a nonlinear version of the Cousin-I problem) 
and a systematic approach for avoiding critical points of a $q$-convex 
Morse exhaustion function on $\Omega$. 

%
%
Earlier constructions of this type, using simpler holomorphic peak functions and higher 
codimension, were given in 1985 by L\o w \cite{Low1985} and Forstneri\v c 
\cite{Forstneric1986TAMS} who showed that every relatively compact 
strongly pseudoconvex domain $D$ in a Stein manifold embeds properly holomorphically 
in a high dimensional Euclidean ball. 
A related result with interpolation on a suitable subset of 
the boundary of $D$ is due to Globevnik \cite{Globevnik1987} (1987).
This peak function technique was inspired by the construction of inner functions on 
the ball of $\C^n$ by L\o w \cite{Low1982} in 1982, based on the work of 
Hakim and Sibony \cite{HakimSibony1982}.

We apply this technique to push the boundary $f_0(bD)\subset \Omega=\C^n\setminus E$
of a holomorphic map $f_0:\bar D\to \C^n$ in Theorem \ref{th:main} 
out of a certain compact convex cap $C$ attached to $E$ along a part of $bC$ contained
in $bE$ and such that the set $E_1=E\cup C$ is convex and has bounded convex 
exhaustion hulls. At the same time, we ensure that the new map $g:\bar D\to \C^n$ still sends 
$\overline {D\setminus K}$ to $\Omega$. For a precise result, 
see Proposition \ref{prop:pushing}. In the next step, we use that 
$\Omega_1=\C^n\setminus E_1$ is an Oka domain (see Corollary \ref{cor:Oka}). 
Since $g(bD)\subset \Omega_1$, we can apply the Oka principle
 (see Theorem \ref{th:Oka}) to approximate $g$ by a holomorphic map 
$f_1:X\to\C^n$ with $f_1(X\setminus D)\subset \Omega_1$. 
Continuing inductively, we obtain a sequence of holomorphic maps $X\to\C^n$ converging to a 
proper map satisfying Theorem \ref{th:main}. The details are given in Section \ref{sec:proof}.

The analogues of Theorem \ref{th:main} and Corollary \ref{cor:main} also hold
for minimal surfaces in $\R^n$.

%
%
\begin{theorem}\label{th:MS}
Let $n\ge 3$, and let $\phi:\R^{n-1}\to\R_+$ be a convex function 
such that the set $\{\phi=0\}$ is nonempty and compact.
Given an open Riemann surface $X$, 
a compact $\Ocal(X)$-convex set $K$ in $X$, and a conformal minimal immersion  
$f_0:U\to\R^n$ from a neighbourhood of $K$ 
with $f_0(bK)\subset \Omega=\{y<\phi(x)\}$, 
we can approximate $f_0$ uniformly on $K$ by proper conformal minimal immersions
$f:X\to \R^n$ (embeddings if $n\ge 5)$ satisfying 
$f(X\setminus \mathring K) \subset \Omega$. 
\end{theorem}

If in addition $\phi$ is of class $\Cscr^1$, strictly convex at infinity, 
and the epigraph $E_\phi=\{y\ge \phi(x)\}$ has BCEH then one can add 
to this statement the interpolation of the map on discrete sets, in analogy to Theorem \ref{th:interpolation}.

Theorem \ref{th:MS} is obtained by following the proof of Theorem \ref{th:main},
replacing Proposition \ref{prop:pushing} by the analogous result obtained by the 
Riemann--Hilbert deformation method for conformal minimal surfaces 
(see \cite{AlarconDrinovecForstnericLopez2015PLMS} or
\cite[Chapter 6]{AlarconForstnericLopez2021}).
Furthermore, it has recently been shown by the authors 
\cite[Corollary 1.5]{DrinovecForstneric2023JMAA}
that the complement of a closed convex set $E\subset \R^n$ $(n\ge 3)$ is flexible
for minimal surfaces (an analogue of the Oka property in complex geometry)
if and only if $E$ is not a halfspace or a slab; clearly this includes all sets with BCEH.

Another method for constructing proper minimal surfaces, which yields
the same result in some examples not covered by Theorem \ref{th:MS}, 
was developed by Alarc\'on and L\'opez \cite{AlarconLopez2012JDG} in 2012.  
They showed that Theorem \ref{th:MS} holds for any wedge domain 
$\Gamma\times \R\subset\R^3$, where $\Gamma\subset \R^2$ 
is an open cone with angle $>\pi$; see \cite[Theorem 5.6]{AlarconLopez2012JDG}. 
The complement of this set is convex but it fails to satisfy the hypotheses of
Theorem \ref{th:MS} due to the presence of lines in the boundary.
An important difference between these two fields, which affects the possible 
construction methods, is that every open Riemann surface admits a proper harmonic 
map to the plane $\R^2$ (see \cite[Theorem I]{AlarconLopez2012JDG}), while only 
few such surfaces admit proper holomorphic maps to $\C$.

The analogue of Problem \ref{prob:sublinear} for minimal surfaces asks whether there is 
a domain in $\R^3$ of the form $\{x_3 < \phi(x_1,x_2)\}$, where $\phi:\R^2\to \R_+$ 
is a function with sublinear growth, 
which contains minimal surfaces of hyperbolic type that are proper in $\R^3$, 
or just a proper hyperbolic end of a minimal surface. In particular, it would be 
interesting to know whether the domain below the upper half of a vertical catenoid
has this property. 
On the other hand, the strong halfspace theorem of Hoffman and Meeks 
\cite{HoffmanMeeks1990IM} says that the only proper minimal surfaces in $\R^3$ 
contained in a halfspace are planes.

%
%
%
%
\section{Pushing a strongly pseudoconvex boundary out of a strictly convex cap}\label{sec:pushing}

Let $O$ be a convex domain in $\C^n$ for some $n>1$. Recall that a continuous function
$\rho:O\to\R$ is said to be {\em strictly convex} if for any pair of points $a,b\in O$ we
have that 
\[
	\rho(ta+(1-t)b) < t\rho(a)+(1-t)\rho(b)\ \ \text{for all $0<t<1$}.
\]

Assume now that $\rho_t:O\to\R$ $(t\in [0,1])$ is a continuous family 
of $\Cscr^1$ functions satisfying the following conditions:

\begin{enumerate}[\rm (a)]
\item 
For every $t\in [0,1]$ the function $\rho_t$ is strictly convex.
Note that $d\rho_t\ne 0$ on $M_t:=\{\rho_t=0\}$.
\item If $0\le s<t\le 1$ then $\rho_t\le 0$ on $M_s$.
\item There is an open relatively compact subset $\omega_0$ of $M_0$ such that 
for every pair of numbers $0\le s<t\le 1$ we have that 
$ 
	M_t \cap M_0 = M_t\cap M_s = M_0\setminus \omega_0.
$ 
\end{enumerate}

This means that the hypersurfaces $M_t$ coincide on the subset $M_0\setminus \omega_0$,
and as $t\in [0,1]$ increases the domains $\omega_t=M_t\setminus M_0 \subset M_t$ 
are pairwise disjoint and move into the convex direction. Each compact set of the form
\begin{equation}\label{eq:Ct}
	C_t= \bigcup_{s\in [0,t]} \overline \omega_s\  \ \ \text{for $t\in[0,1]$} 
\end{equation}
is called a {\em strictly convex cap with the base} $\omega_0$. 
Note that $bC_t=\overline \omega_0\cup \overline\omega_t$,
$C_t$ is strictly convex along $\omega_t$, strictly concave along $\omega_0$, and it has 
corners along $\overline\omega_0\cap \overline\omega_t$. As $t\in[0,1]$ increases to $1$, 
the caps $C_t$ monotonically increase to $C_1$ and they share the same base $\omega_0$. Likewise, for any $0\le s<t\le 1$ the set $C_{s,t}=\bigcup_{u\in [s,t]} \overline \omega_u$ 
is a strictly convex cap with the base $\omega_s$. The sets
\begin{equation}\label{eq:Et}	
	E_t=\{z\in O:\rho_t(z) \le 0\}\ \ \text{for $t\in [0,1]$} 
\end{equation}
are strictly convex along $bE_t=\{\rho_t=0\}$, 
they form a continuously increasing family in $t$, and 
\[
	E_t = E_0 \cup C_t\ \ \text{for every $t\in [0,1]$}. 
\]
Under these assumptions, we have the following result. 

%
%
\begin{proposition}\label{prop:pushing}
Let $D$ be a smoothly bounded, relatively compact, 
strongly pseudoconvex domain in a Stein manifold $X$ with $\dim X<n$. 
Let the sets $E_t\subset  O\subset\C^n$ $(t\in [0,1])$ be given by \eqref{eq:Et}, 
and let $f_0:\bar D\to  O$ be a map of class $\Ascr(\bar D)$ such that 
$f_0(bD)\cap E_0 = \varnothing$. Given a compact set $K\subset D$ such that 
$f_0(\overline{D\setminus K}) \cap  E_0 = \varnothing$ and 
a number $\epsilon>0$, there is a map $f:\bar D\to O$ of class $\Ascr(\bar D)$ 
satisfying the following conditions:
\begin{enumerate}[\rm (i)]
\item $f(bD)\cap E_1=\varnothing$,
\item $f(\overline{D\setminus K}) \cap  E_0 = \varnothing$, and 
\item $\max_{x\in K}|f(x)-f_0(x)|<\epsilon$.  
\end{enumerate}
\end{proposition}

Recall that a map $f:\bar D\to  O$ 
is said to be of class $\Ascr(\bar D)$ if it is continuous on $\bar D$ and holomorphic on $D$.
In our application of Proposition \ref{prop:pushing} in the proof of Theorem \ref{th:main}, 
the set $O$ will be a ball (or the entire Euclidean space) 
and the hypersurfaces $M_t=\{\rho_t=0\}=bE_t$ will be convex graphs
over the coordinate hyperplane $\C^{n-1}\times \R\subset \C^n$. 

In the proof of Proposition \ref{prop:pushing} we shall need the following lemma. 

%
%
%
%
\begin{lemma} \label{lem:main}
Assume that $O$ is a convex open subset of $\C^n$ for $n>1$, $L$ is a compact 
subset of $O$, and $\rho:O\to \R$ is a $\Cscr^1$ smooth strictly convex function.
Then there is a number $\delta>0$ with the following property.
If $D$ is a smoothly bounded strongly pseudoconvex domain in a Stein manifold 
$X$ of dimension $\dim X =m<n$, $K$ is a compact subset of $D$, and $f: \bar D\to O$ 
is a map of class $\Ascr(\bar D)$ such that  
\begin{equation}\label{eq:aboutf}
  	\rho(f(z))> -\delta\ \ \text{for all $z\in bD$} 
	\quad \text{and}\quad \rho(f(z))>0 \ \ \text{if $z\in bD$ and $f(z)\notin L$}, 
\end{equation}
then given $\eta >0$ there is  a map $g:\bar D\to O$ of class $\Ascr(\bar D)$ 
satisfying the following conditions:
\begin{enumerate}[\rm (i)] 
\item  $\rho(g(z))>0$ for $z\in bD$,
\item  $\rho(g(z))>\delta$ for those $z\in bD$ for which $g(z)\in L$,
\item  $\rho(g(z))>\rho(f(z))-\eta$ for $z\in \overline{D\setminus K}$, and
\item  $|f(z)-g(z)|<\eta$ for $z\in K$. 
\end{enumerate}
\end{lemma}

For $m=1$, i.e., when $D$ is a finite bordered Riemann surface, 
this is a simplified version of \cite[Lemmas 6.2 and 6.3]{DrinovecForstneric2007DMJ}, 
which is proved by using approximate solutions of a Riemann--Hilbert boundary value problem. 
This method was employed in several earlier papers mentioned in \cite{DrinovecForstneric2007DMJ}.
When $\rho$ is strictly convex, $\Cscr^1$ smoothness suffices since 
in the proof we may take a continuous family of tangential linear discs to the sublevel set of $\rho$.

For $m\ge 2$, Lemma \ref{lem:main} is a simplified and slightly modified
version of \cite[Lemma 5.3]{DrinovecForstneric2010AJM}.
Besides the fact that we are considering single maps $\bar D\to O$ 
instead of sprays of maps, the only difference is that the assumption in 
\cite[Lemma 5.3]{DrinovecForstneric2010AJM} that the set 
$\{\rho=0\}$ is compact is replaced by the assumption \eqref{eq:aboutf} saying that 
$\rho(f(z))$ for $z\in bD$ may be negative only if $f(z)$ lies in the compact set 
$L\subset O$. This hypothesis ensures that the lifting for a relatively big amount 
(the role of the constant $\delta$) only needs to be made on a compact subset of $O$, 
while elsewhere it suffices to pay attention not to decrease $\rho\circ f$ 
by more than a given amount and to approximate sufficiently closely on $K$ 
(the role of the constant $\eta$). The proof requires only a minor adaptation of
\cite[proof of Lemma 5.3]{DrinovecForstneric2010AJM}, using its local version 
\cite[Lemma 5.2]{DrinovecForstneric2010AJM} in a finite induction with respect to a 
covering of $bD$ by small open sets on which there are good systems of local holomorphic 
peak functions. In fact, Lemma \ref{lem:main} corresponds to a simplified version of 
\cite[Sublemma 5.4]{DrinovecForstneric2010AJM}, which explains how to lift the image of 
$bD$ with respect to $\rho$ for a sufficiently large amount at those points in $bD$ which 
the map $f$ sends to a certain coordinate chart $U_i$ in the target manifold. 
In our situation, the role of $U_i$ is played by an open relatively compact neighbourhood
of the set $L\cap\{\rho=0\}$ in $O$, and there is no need to use the rest of the proof of 
\cite[Lemma 5.3]{DrinovecForstneric2010AJM}. 

%
%
\begin{proof}[Proof of Proposition \ref{prop:pushing}]
For $t\in [0,1]$ let $\delta_t>0$ be a number for which the conclusion of 
Lemma \ref{lem:main} holds for the function $\rho_t$ and the compact set $L=C_1$ 
(see \eqref{eq:Ct}). The open sets 
\[
	U_t=\{z\in O: -\delta_t<\rho_t(z)<\delta_t\}\ \ \text{for $t\in [0,1]$}
\]
form an open covering of $C_1$, so there exists a finite subcovering 
$\{U_{t_j}\}$ for $0\le t_1<t_2<\ldots <t_k\le 1$. Applying Lemma \ref{lem:main}
we inductively find maps $f_1,\ldots, f_k\in\Ascr(\bar D)$ such that for every 
$j=1,\ldots,k$ we have that 
\begin{enumerate}[\rm (a)]
\item $f_j(bD)\cap E_{t_j}=\varnothing$ (where $E_t$ is given by \eqref{eq:Et}), 
\item $f_j(\overline{D\setminus K})\cap\overline E_0=\varnothing$, and 
\item $|f_j-f_{j-1}|<\epsilon/k$ on $K$. 
\end{enumerate}
Note that conditions (a) and (b) hold for $f_0$ and (c) is void.
Assume inductively that for some $j\in\{1,\ldots,k\}$ we have maps 
$f_0,\ldots,f_{j-1}$ satisfying these conditions. 
Applying Lemma \ref{lem:main} with $f=f_{j-1}$ and taking $f_j=g$,
condition (a) follows from part (i) in Lemma \ref{lem:main}, (b) follows from (ii) 
provided that the number $\eta>0$ in Lemma \ref{lem:main} is chosen small enough, 
and (c) follows from (iii) in Lemma \ref{lem:main} provided that $\eta\le \epsilon/k$ . 
This gives the map $f_j$ satisfying conditions (a)--(c) and the induction may continue.
The map $f=f_k$ then satisfies the proposition.
\end{proof}

\begin{remark}
Proposition \ref{prop:pushing} also holds, with the same proof, if $\rho_t$ $(t\in[0,1])$ 
are strongly plurisubharmonic functions of class $\Cscr^2$ satisfying $d\rho_t\ne 0$ on
$M_t=\{\rho_t=0\}$. Indeed, the results from \cite{DrinovecForstneric2010AJM},
which are used in the proof, pertain to this case. In the present paper
we shall only use the convex case under $\Cscr^1$ smoothness, which comes
naturally in the construction. 
\end{remark}

%
%
%
%
\section{Closed convex sets with BCEH}\label{sec:BCEH}

In the context of convex analysis, closed unbounded convex sets that share several
important properties with compact convex sets were studied by 
Gale and Klee \cite{GaleKlee1959} in 1959. They introduced the class of continuous sets, 
and we show that this class coincides with the class 
of sets having BCEH, introduced in Definition \ref{def:BCEH}; 
see Proposition \ref{prop:BCEHcontinuous}. We then develop further
properties of these sets which are relevant to the proof of our main theorems.

By a {\em ray} in $\R^n$, we shall mean a closed affine halfline. 
Let $E$ be a closed convex subset of $\R^n$. A \emph{boundary ray} of $E$ 
is a ray contained in the boundary of $E$. An \emph{asymptote} of $E$ is 
a ray $L\subset \R^n\setminus E$ such that $\dist(L,E)=\inf\{|x-y|:x\in L,\ y\in E\}=0$.
The function 
\[
	\sigma: \{u\in\R^n : |u|=1\}\to \R\cup\{+\infty\}, \quad 
	\sigma(u)=\sup\{x\,\cdotp u : x\in E\} 
\]
is called the \emph{the support function} of $E$. (Here, $x\,\cdotp u$ denotes  
the Euclidean inner product.) A closed convex set $E$ is said to be {\em continuous} 
in the sense of Gale and Klee \cite{GaleKlee1959} if the support function of $E$
is continuous. Note that every compact convex set is continuous.

The following result is a part of \cite[Theorem 1.3]{GaleKlee1959} due to 
Gale and Klee; we only list those conditions that will be used. 
The last item (iv) uses also \cite[Theorem 1.5]{GaleKlee1959}. 

%
%
\begin{theorem}\label{th:GaleKlee}
For a closed convex subset $E$ in $\R^n$ the following conditions are equivalent:
\begin{enumerate}[\rm (i)] 
\item $E$ is continuous. 
\item $E$ has no boundary ray nor asymptote. 
\item For each point $p\in \R^n$ the convex hull $\Conv(E\cup \{p\})$ is closed.
\item For every compact convex set $K\subset\R^n$ the set $\Conv(E\cup K)$ 
is closed.
\end{enumerate}
\end{theorem}

Condition (iii) implies that the closed convex hull $\overline{\Conv}(E\cup\{p\})$ 
is the union of the line segments connecting $p$ to the points in $E$. 
It also shows that an unbounded continuous closed convex subset $E$ of $\R^n$ 
is not contained in any affine hyperplane. 

Let us record the following observation which will be used in the sequel.

%
%
\begin{lemma}\label{lem:L}
Let $E\subset \R^n$ be a closed convex set, $p\in \R^n\setminus E$,
and $L \subset \R^n$ be an affine subspace containing $p$. Then,
$\Conv(E\cup\{p\}) \cap L =\Conv((E\cap L)\cup \{p\})$.
\end{lemma}

\begin{proof}
Set $E'=E\cap L$. It is obvious that 
$\Conv(E'\cup \{p\}) \subset \Conv(E\cup\{p\})\cap L$. Conversely,
since $E$ is convex, every point $q\in \Conv(E\cup\{p\})$ belongs to a line segment 
from $p$ to a point $q'\in E$. If in addition $q\in L$ and $q\ne p$ then 
$q'\in E'$, and hence $q\in \Conv(E' \cup \{p\})$.
\end{proof}

%
%
\begin{proposition}\label{prop:BCEHcontinuous}
A closed convex set $E\subset \R^n$ has BCEH if and only if it is continuous
in the sense of Gale and Klee \cite{GaleKlee1959}.
\end{proposition}

\begin{proof}
Since all closed bounded convex sets have BCEH and are continuous, 
it suffices to consider the case when the set $E$ is unbounded.

If $E$ is not continuous then by Theorem \ref{th:GaleKlee} it has a boundary ray or 
an asymptote. Denote it by $L$, and let $\ell$ be the affine line containing $L$. 
Pick any affine $2$-plane $H\subset\R^n$ containing $\ell$. 
There is a point $p\in H\setminus (\ell\cup E)$.
By considering rays from $p$ to points $q\in E$ approaching $L$ and going to infinity
(if $L$ is a boundary ray, we can choose points $q\in L$), we see that the closure 
of the set $h(E,p)=\Conv(E \cup\{p\})\setminus E$ contains the parallel translate 
$L' \subset H^+$ of $L$ passing through $p$, so $h(E,p)$ is unbounded and hence
$E$ does not have BCEH.

Assume now that $E$ is a continuous and let us prove that it has BCEH. 
We need to show that for any closed ball $B\subset\R^n$ the set 
$h(E,B)=\Conv(E\cup B)\setminus E$ is bounded. Assume to the contrary that there is 
a sequence $x_m\in h(E,B)$ with $|x_m|\to\infty$ as $m\to\infty$. 
Since the sets $E$ and $B$ are convex, we have that  
\[
	\text{$x_m=t_mb_m+(1-t_m) e_m$\ \ for $t_m\in [0,1]$, 
	$b_m\in B$, $e_m\in E$, and $m\in\N$.} 
\]
 Note that $(1-t_m)|e_m|\to\infty$ as $m\to\infty$. By compactness of the respective sets 
we may assume, passing to a subsequence, that $e_m\ne 0$ for all $m$
and the sequences $t_m$, $b_m$, and $\frac 1 {|e_m|}e_m$ are convergent. 
Denote their respective limits by $t$, $b$, and $f$.
We have that
\[
	x_m=t_mb_m+(1-t_m) e_m
	=b_m+(1-t_m) |e_m|\left( \frac{e_m}{|e_m|} -\frac{b_m}{|e_m|}\right)
	=b_m+(1-t_m) |e_m| f_m 
\]
where $f_m=\bigl( \frac{e_m}{|e_m|} -\frac{b_m}{|e_m|}\bigr)$. 
Note that $\lim_{m\to\infty}f_m=f$. Pick a number $\alpha\ge 0$ and set $p=b+\alpha f$. 
If $m$ is large enough then $(1-t_m) |e_m|>\alpha$, so the point 
$
	y_m=b_m+\alpha f_m
$ 
lies on the line segment connecting $b_m$ and $x_m$. 
Since $x_m\in \Conv(E\cup \{b_m\})$, it follows that $y_m \in \Conv(E\cup \{b_m\})$. 
Note that the sequence $y_m$ converges to $p$. Since $E$ is continuous, 
$\Conv(E\cup\{b\})$ is closed by Theorem \ref{th:GaleKlee}, so 
$p=\lim_{m\to\infty}y_m \in\Conv(E\cup\{b\})$. Since this holds for every $\alpha\ge 0$, 
the ray $L=\{b+\alpha f : \alpha\in[0,\infty)\}$ lies in $\Conv(E\cup\{b\})$.
By Lemma \ref{lem:L} there is 
$\alpha_0\in[0,\infty)$ such that the ray $L'=\{b+\alpha f : \alpha\ge \alpha_0\}$ lies in $E$. 
Since $E$ is continuous, $L$ is not a boundary ray of $E$ by Theorem \ref{th:GaleKlee},
thus $L$ contains a point $q=b+\alpha_1 f \in E\setminus bE$ 
for some $\alpha_1 \ge \alpha_0$. Choose a neighbourhood $U_q\subset E$ of $q$.
For any large enough $m$ we then have $p_m:=b_m+\alpha_1 f_m \in U_q$.
Let $L_m=\{b_m+\alpha f_m : \alpha\ge 0\}$. Note that 
$L_m\cap \Conv(E\cup\{b_m\}) = \Conv((L_m\cap E)\cup\{b_m\})$ by Lemma \ref{lem:L}.
However, for $m$ large enough the point $x_m\in L_m$ lies on 
the opposite side of $p_m$ than $b_m$, so $x_m$ belongs to $L_m\cap \Conv(E\cup\{b_m\})$ 
but not to $\Conv((L_m\cap E)\cup\{b_m\})$. This contradiction proves that $E$ has BCEH.
\end{proof}

Given a function $\phi:\R^{n-1}\to\R$, the {\em epigraph} of $\phi$ is the set 
\begin{equation}\label{eq:epi-phi}
	E=E_\phi=\{(x,y)\in\R^{n-1}\times \R : y \ge \phi(x)\}.
\end{equation} 
Note that a function is convex if and only if its epigraph is convex.

%
%
\begin{proposition}\label{prop:epigraph}
If $E\subsetneq \R^n$ is a closed unbounded convex set with BCEH then
\begin{enumerate}[\rm (i)]
\item $E$ does not contain any affine real line, and 
\item for every affine line $\ell$ intersecting $E$ in a ray and any hyperplane $H$
transverse to $\ell$, $E$ is the epigraph of a convex function on $H$.
In particular, there are affine coordinates $(x,y)$ on $\R^n$ in which $E$ is of the form
\eqref{eq:epi-phi} for a convex function $\phi:\R^{n-1}\to \R_+$ satisfying
\begin{equation}\label{eq:proper}
	\liminf_{|x|\to+\infty} \frac{\phi(x)}{|x|} >0.
\end{equation}
\end{enumerate}
\end{proposition}

The condition \eqref{eq:proper} says that $\phi$ grows at least linearly at infinity.
We show in Example \ref{ex:lineargrowth} that linear growth is possible.

\begin{proof}
\rm (i)
Assume that $\ell\subset E$ is an affine line and let us prove that $E$
does not have BCEH. Since $E$ is a proper subset of $\R^n$,
there is a parallel translate $\ell'$ of $\ell$ which is not contained in $E$, 
and hence $\ell'\setminus E$ contains a ray $L$. Let $p$ be the endpoint of $L$,
and let $p'\in L$ be an arbitrary other point. Since $E\cap L=\varnothing$,
there is a ball $B$ around $p'$ such that $\Conv(B\cup\{p\}) \cap E=\varnothing$.
Clearly, there is a point $q\in B$ such that the ray $L_q$ with the endpoint $p$ and 
containing $q$ intersects the line $\ell$, so the line segment from $p$ to $q$
belongs to $\Conv(E\cup\{p\})\setminus E = h(E,p)$. By moving $p'\in L$ to infinity
we see that $h(E,p)$ is unbounded, so $E$ does not have BCEH.

\rm (ii) 
Since $E$ is unbounded, it contains a ray $L$. Denote by $\ell$ the affine line 
containing $L$. Let $\ell'$ be any parallel translate of $\ell$. Since $E$ contains no affine lines
by part (i), there is a point $p \in \ell'\setminus E$. 
The closed convex hull of the union of $L$ and $p$ contains 
the parallel translate $L'\subset \ell'$ of $L$ passing through $p$.
Since $E$ has BCEH, we conclude that $L'\subset \Conv(E\cup\{p\})$ and 
$L'\setminus E$ is bounded. Since $E\cap L'$ is convex, 
$L'\cap E$ is a closed ray with the endpoint on $bE$.
This shows that $E$ is a union of closed rays contained in parallel translates of the line $\ell$,
so it is an epigraph of a convex function defined on any hyperplane $H\subset\R^n$
transverse to $\ell$. Choosing $H$ such that $H\cap E=\varnothing$ there are  
affine coordinates $(x,y)$ on $\R^n$ with $H=\{y=0\}$ and $\ell=\{x=0\}$. 
In these coordinates, $E$ is of the form \eqref{eq:epi-phi} for a positive convex function $\phi$. 

Finally, if condition \eqref{eq:proper} fails then there is a sequence $(x_k,y_k)\in E$ with
$|x_k|\to+\infty $ and $y_k/|x_k|\to 0$ as $k\to \infty$. The union of the 
line segments $L_k$ connecting $p=(0,-1)\in\R^{n-1}\times\R$ to $(x_k,y_k)$, 
intersected with the lower halfspace $y\le 0$, is then an unbounded
subset of $h(E,p)=\Conv(E\cup \{p\})\setminus E$, contradicting the assumption 
that $E$ has BCEH. 
\end{proof}

%
%
\begin{remark}\label{rem:proper}
The growth condition \eqref{eq:proper} for an epigraph 
can always be achieved in suitable linear coordinates 
(even without the BCEH property) if there is a supporting hyperplane $H\subset \R^n$ 
for $E$ such that the set $E\cap H$ is nonempty and compact. Indeed, we may then choose
coordinates $(x,y)$ on $\R^n$ such that $H=\{y=0\}$, $E\subset \{y\ge 0\}$,
and $0\in E$. If the condition \eqref{eq:proper} fails, 
there is a sequence $(x_k,y_k)\in E$ with $|x_k|\to+\infty $ and $y_k/|x_k|\to 0$ 
as $k\to \infty$. After passing to a subsequence, a ray in $E\cap H$ lies
in the closure of the union of the line segments $L_k\subset E$ connecting the origin to 
$(x_k,y_k)$, contradicting the assumption that the latter set is compact. 
\end{remark}

%
%
\begin{corollary}\label{cor:Oka}
If $E$ is a closed convex set in $\C^n$ $(n>1)$ having BCEH 
then $\C^n\setminus E$ is Oka. 
\end{corollary}

\begin{proof}
By Proposition \ref{prop:epigraph} the set $E$ does not contain any affine real line, 
and hence $\C^n\setminus E$ is Oka by \cite[Theorem 1.8]{ForstnericWold2022Oka}.
\end{proof}

The following lemma shows that the BCEH condition is stable under uniform approximation.

%
%
\begin{lemma}\label{lem:approximation}
Assume that $\phi:\R^{n-1}\to\R$ is a convex function whose epigraph
$E_\phi$ \eqref{eq:epi-phi} has BCEH. Then for any $\epsilon>0$ 
and convex function $\psi:\R^{n-1}\to\R$ satisfying $|\phi-\psi|<\epsilon$
the epigraph $E_\psi$ also has BCEH.
\end{lemma}

\begin{proof}
If $E_\psi$ fails to have BCEH then by Theorem \ref{th:GaleKlee} and
Proposition \ref{prop:BCEHcontinuous} it has a boundary ray or an asymptote, $L$.
Since $\dist(L,E_\psi)=0$ and $E_\psi$ is convex, $\dist(x,E_\psi)$ converges to zero 
as $x\in L$ goes to infinity. Thus, by making $L$ shorter if necessary, we have that 
$L\subset E_{\phi-2\epsilon} \setminus E_{\phi+2\epsilon}$. Hence, 
$L$ lies out of $E_{\phi+2\epsilon}$ but the vertical translation of $L$ for $4\epsilon$
pushes it in $E_{\phi+2\epsilon}$. Since $E_{\phi+2\epsilon}$, being a translate of 
$E_\phi$, has BCEH, this contradicts Proposition \ref{prop:epigraph} (ii).
The contradiction shows that $E_\psi$ has BCEH as claimed. 
\end{proof}

We now give a differential characterization of the BCEH property of an epigraph 
\eqref{eq:epi-phi}.

%
%
\begin{proposition}\label{prop:BCEH}
If $\phi:\R^{n-1}\to \R$ is a convex function of class $\Cscr^1$ satisfying 
condition \eqref{eq:proper}, then the epigraph $E=\{(x,y)\in\R^n:y\ge \phi(x)\}$ 
has BCEH if and only if
\begin{equation}\label{eq:tangentline}
	\lim_{|x|\to\infty} |x| \left(1-\frac{\phi(x)}{x\,\cdotp \nabla\phi(x)} \right)
	= +\infty.
\end{equation}
\end{proposition}

\begin{proof}
We first consider the case $n=2$. Then, $x$ is a single variable and \eqref{eq:tangentline} 
is equivalent to 
\begin{equation}\label{eq:tangentline2}
	\lim_{x\to+\infty} \left(x - \frac{\phi(x)}{\phi'(x)}\right) = +\infty
	\quad \text{and}\quad
	\lim_{x\to -\infty} \left(x - \frac{\phi(x)}{\phi'(x)}\right) = -\infty.
\end{equation}
For every $x\in \R$ such that $\phi'(x)\ne 0$ the number
\begin{equation}\label{eq:xi}
	\xi(x) = x - \frac{\phi(x)}{\phi'(x)}
\end{equation}
is the first coordinate of the intersection of the tangent line to the graph of $\phi$ 
at the point $(x,\phi(x))$ with the first coordinate axis $y=0$. By \eqref{eq:proper} and convexity 
of $\phi$ we have that $|\phi'(x)|$ is bounded away from zero for all sufficiently 
big $|x|$. This shows that conditions \eqref{eq:tangentline2} are invariant under
translations, so we may assume that $\phi\ge 0$ and $\phi(0)=0$. It is easily seen that 
the function $\xi$ is increasing. If $\phi$ is of class $\Cscr^2$, 
we have that $\xi'(x)=\phi(x)\phi''(x)/\phi'(x)^2 \ge 0$. 

Assume now that conditions \eqref{eq:tangentline2} hold.
Pick a pair of sequences $a_j < b_j$ in $\R$ with $\lim_{j\to\infty}a_j=-\infty$ and 
$\lim_{j\to\infty}b_j=+\infty$. The intervals $I_j=[\xi(a_j),\xi(b_j)]$ then 
increase to $\R$ as $j\to\infty$. We identify $I_j$ with $I_j\times\{0\}\subset\R^2$.
Since $\phi$ is convex, its epigraph lies above the tangent line at any point. 
It follows that the set $h(E,I_j)$ (see \eqref{eq:h}) is the bounded region in 
$\R\times \R_+$ whose boundary consists of $I_j$, the two line segments $L_j$ and $L'_j$ 
connecting the endpoints $(\xi(a_j),0)$ and $(\xi(b_j),0)$ of $I_j$ to the respective points 
$A_j=(a_j,\phi(a_j))$ and $B_j=(b_j,\phi(b_j))$ on $bE$, and the graph of $\phi$ over $[a_j,b_j]$. 
The supporting lines of $L_j$ and $L'_j$ intersect at a point $C_j$ in the lower halfspace
$y<0$, and we obtain a closed triangle $\Delta_j$ with the endpoints $A_j,B_j$, and $C_j$.
Note that $\Delta_j\cap(\R\times \{0\}) = I_j$. Since $\phi$ grows at least 
linearly (see \eqref{eq:proper}), the triangles $\Delta_j\subset \R^2$ 
exhaust $\R^2$ as $j\to\infty$, and the set $h(E,\Delta_j)$ \eqref{eq:h} 
is bounded for every $j$. Hence, $E$ has BCEH. 
This argument furthermore shows that for any point $p=(0,-c)\notin E$ 
there is a unique pair of tangent lines to $bE$ passing through $p$ such that,
denoting by $q_1,q_2\in bE$ the respective points where these lines intersect $bE$,
the convex hull $\Conv(E\cup\{p\})$ is the union of $E$ and the triangle with vertices
$p,q_1,q_2$.

Conversely, if \eqref{eq:tangentline} fails then it is easily seen that $E$ has a boundary ray 
or an asymptote, so it does not have BCEH. We leave the details to the reader.

The case with $n\ge 3$ now follows easily. Pick a unit vector $v\in \R^{n-1}$,
$|v|=1$, and let $L_v$ denote the $2$-plane in $\R^n$ passing through the 
origin and spanned by $v$ and $e_n=(0,\ldots,0,1)$. Then, 
$E_v:=E\cap L_v= \{(t,y)\in\R^2: y \ge \phi(tv)\}$ and the first condition 
in \eqref{eq:tangentline2} reads
\begin{equation}\label{eq:tangentline3}
	\lim_{t\to +\infty} \left(t-\frac{\phi(tv)}{\sum_{j=1}^{n-1} v_j 
	\frac{\di \phi}{\di x_j}(tv)} \right)
	= +\infty.
\end{equation}
Writing $x=tv$ with $t\ge 0$ and $v=x/|x|$, this is clearly equivalent to 
\eqref{eq:tangentline}. As before, let $p=(0,\ldots,0,-c)\notin E$.
If \eqref{eq:tangentline} holds then $\Conv(E_v\cup\{p\})\subset L_v$ is 
obtained by adding to $E_v$ the triangle in $L_v$ obtained by the two tangent lines to 
$bE_v$ passing through $p$ as described in the case $n=2$. The sizes of 
these triangles are uniformly bounded with respect to the 
direction vector $|v|=1$, and condition \eqref{eq:proper} implies 
that these triangles increase to $L_v$ as $c\to +\infty$, uniformly with respect to $v$.
Since $\bigcup_{|v|=1}L_v=\R^n$, Lemma \ref{lem:L} shows that 
\[
	\Conv(E\cup \{p\}) = \bigcup_{|v|=1} \Conv(E_v\cup \{p\}),
\]
and hence $E$ has BCEH. The converse is seen as in the special case $n=2$.
\end{proof}

%
%
\begin{corollary}\label{cor:derivative}
If $\phi:\R^{n-1}\to\R_+$ is a convex function of class $\Cscr^1$ such that 
\[
	\lim_{|x|\to+\infty} \frac{x\,\cdotp \nabla\phi(x)}{|x|} =+\infty, 
\]
then the epigraph $E=\{(x,y) \in \R^n : y\ge \phi(x)\}$ has BCEH.
\end{corollary}

\begin{proof}
By restricting to planes as in the above proof, it suffices to consider
the case $n=2$. We may assume that $\phi\ge 0$ and $\phi(0)=0$.
Since $\phi$ is convex, $g(x)=\phi'(x)$ is an increasing function 
and the above condition reads $\lim_{x\to\pm\infty} g(x)=\pm\infty$. 
For any $x_0>0$ and $x\ge x_0$ we have that
\[
	\xi(x):= x - \frac{1}{g(x)} \int_0^x g(t)dt 
	=\int_0^{x} \Big(1-\frac{g(t)}{g(x)}\Big) dt
	\ge \int_0^{x_0} \Big(1-\frac{g(t)}{g(x)}\Big) dt.
\]
Letting $x\to+\infty$ we have that $\frac{g(t)}{g(x)}\to 0$ uniformly 
on $t\in [0,x_0]$, and hence the last integral converges to $x_0$. 
Letting $x_0\to\infty$ we see that $\lim_{x\to+\infty}\xi(x)=+\infty$. 
The analogous argument applies when $x\to-\infty$. Hence, conditions 
\eqref{eq:tangentline} hold and therefore $E$ has BCEH.
\end{proof}

%
%
\begin{example}\label{ex:lineargrowth}
There exist convex epigraphs \eqref{eq:epi-phi} 
having BCEH where the function $\phi$ grows linearly, although it cannot 
be too close to linear near infinity in the absence of boundary rays and asymptotes. 
We give such an example in $\R^2$.
Let $g:\R\to (-1,1)$ be an odd, continuous, increasing function
with $\lim_{x\to+\infty} g(x)=1$ and $\int_0^\infty (1-g(x))dx = +\infty$.
(An explicit example is $g(x) = \frac{2}{\pi}\mathrm{Arctan}\,x$.) 
Its integral $\phi(x)=\int_0^x g(t)dt$ for $x\in\R$ then clearly satisfies 
$\phi(x)\ge 0$, $\phi'(x)=g(x)\in (-1,+1)$ (hence $\phi$ grows linearly), and $\phi$ is convex.
We now show that  \eqref{eq:tangentline} holds. Let $x>0$ be large enough 
so that $g(x)>0$. We have that
\[
	\xi(x) = x - \frac{1}{g(x)} \int_0^x g(t)dt 
	=\int_0^{x} \Big(1-\frac{g(t)}{g(x)}\Big) dt.
\]
Fix $x_0>0$ and let $x\ge x_0$. Then, $\xi(x)\ge \int_0^{x_0} (1-g(t)/g(x))dt$.
Since $\lim_{x\to+\infty} g(x)=1$ and $\xi$ is increasing for large enough $|x|$, 
it follows that 
$
	\lim_{x\to+\infty} \xi(x) \ge \int_0^{x_0} (1-g(t))dt.
$
Sending $x_0\to +\infty$ gives 
$
	\lim_{x\to+\infty} \xi(x) \ge \int_0^\infty (1-g(t))dt =+\infty.
$
Similarly we see that $\lim_{x\to -\infty} \xi(x) =-\infty$. Thus, \eqref{eq:tangentline} holds, 
and hence the epigraph of $\phi$ has BCEH. 
\end{example}

By using the idea in the above example we now prove the following
approximation result, which extends Theorem \ref{th:main}
to a much bigger class of convex epigraphs (see Corollary \ref{cor:main}).

%
%
\begin{proposition}\label{prop:approximation} 
Assume that $\phi:\R^{n-1}\to\R_+$ is a convex function such that the set $\{\phi=0\}$ 
is nonempty and compact. Given numbers $\epsilon>0$ (small) and $R>0$ (big) there is a 
smooth convex function $\psi:\R^{n-1}\to\R$ such that $\psi<\phi$ on $\R^{n-1}$, 
$\phi(x) -\psi(x) < \epsilon$ for all $|x|\le R$, and the epigraph $E_\psi=\{y\ge \psi\}$ has BCEH.
\end{proposition}

\begin{proof}
By Remark \ref{rem:proper} the function $\phi$ grows at least linearly 
near infinity (see \eqref{eq:proper}). Set 
\begin{equation}\label{eq:A}
	A =\liminf_{|x|\to\infty} \frac{\phi(x)}{|x|} >0. 
\end{equation}
Since the set $\phi=0$ does not contain any affine line,  
Azagra's result \cite[Theorem 1 and Proposition 1]{Azagra2013} 
implies that for every $\epsilon>0$ there is a smooth 
strictly convex function $\psi$ on $\R^{n-1}$ satisfying $\phi-\epsilon<\psi <\phi$. 
Replacing $\phi$ by $\psi-\min_x \psi(x)\ge 0$ we may therefore assume that 
$\phi$ is smooth. By increasing the number $R>0$ if necessary,
we may assume that 
\begin{equation}\label{eq:A2}
	\frac{\phi(x)}{|x|} \ge \frac{A}{2} \ \ \ \text{for all}\ |x|\ge R. 
\end{equation}
Pick a number $r\in (0,1)$ close to $1$ such that 
\begin{equation}\label{eq:epsilon}
	(1-r) \sup_{|x|\le R} \phi(x) <\epsilon.
\end{equation} 
Choose a smooth increasing function $h:\R\to \R_+$ such that
\[
	\text{$h(t)=0$ for $t\le R$}, \quad \lim_{t\to+\infty} h(t)=1,
	\quad \text{and}\quad \int_0^\infty (1-h(t))dt=+\infty. 
\] 
(We can take a smoothing of the $\mathrm{Arctan}$ function used in 
Example \ref{ex:lineargrowth}.) Set 
\[
	H(x) = \int_0^{|x|} h(s)ds\quad \text{for $x\in\R^{n-1}$}.
\]
Clearly, $H\ge 0$ is a radially symmetric smooth convex function 
that vanishes on $|x|\le R$ and satisfies $H(x)\le |x|$ for all $x\in\R^{n-1}$.
With $A$ and $r$ as in \eqref{eq:A} and \eqref{eq:epsilon} we set
\[
	\delta= \frac{A(1-r)}{2}.
\] 
We claim that the function 
\[
	\psi(x) = r\phi(x) + \delta H(x) \quad \text{for $x\in \R^{n-1}$}
\]
satisfies the conditions in the theorem. Clearly, $\psi\ge r\phi$ is a smooth convex function.
For $|x|\le R$ we have $H(x)=0$, so $\psi(x)=r\phi(x) \le \phi(x)$ and 
$\phi(x)-\psi(x) = (1-r)\phi(x)<\epsilon$ by \eqref{eq:epsilon}.
If $|x|>R$ then $\phi(x)/|x|\ge A/2$ by \eqref{eq:A2} and $H(x)<|x|$, which implies 
\[
	\frac{\psi(x)}{|x|} \le r \frac{\phi(x)}{|x|} + \delta \le  \frac{\phi(x)}{|x|}.
\]
Indeed, we have that 
$\frac{\phi(x)}{|x|}- r\frac{\phi(x)}{|x|} = (1-r) \frac{\phi(x)}{|x|} \ge \frac{A(1-r)}{2}=\delta$.
Hence, $\psi\le \phi$ on $\R^{n-1}$.

It remains to show that the epigraph $E_\psi$ satisfies BCEH. 
We shall verify \eqref{eq:tangentline}, which is equivalent to
\eqref{eq:tangentline3} with uniform convergence with respect to the 
vector $v=x/|x|$. Write
\[
	g_v(t) = r \frac{\di \phi (tv)}{\di t},\quad 
	k(t)=\delta h(t) ,\quad 
	\tilde g_v(t) = \frac{\di \psi(tv)}{\di t}= g_v(t) + k(t).
\]
The quantity in \eqref{eq:tangentline3} associated to the function $\psi$ is given by
\begin{eqnarray*}
	\xi_v(t) &=& t - \frac{\psi(tv)}{\tilde g_v(t)} = 
	\int_0^t \left(1-\frac{g_v(s)+k(s)}{g_v(t)+k(t)}\right) ds \\
	&=& \int_0^t \frac{g_v(t)-g_v(s)}{g_v(t)+k(t)}ds 
	+ \int_0^t \frac{k(t)-k(s)}{g_v(t)+k(t)}ds \\
	&\ge&  \int_0^t \frac{g_v(t)-g_v(s)}{g_v(t)+\delta}ds +
	\int_0^t \frac{k(t)-k(s)}{g_v(t)+\delta}ds, 
\end{eqnarray*}
where the last inequality holds since the functions $g_v$ and $k$ are nonnegative
and increasing and $k<\delta$. Pick $c>0$. We will show that for large
enough $t>0$ and any unit vector $v\in\R^{n-1}$ 
the above expression is bigger than or equal to $c$.
Choose positive numbers $t_0,a,t_1$ as follows:
\[
	t_0=3 c,\quad a= \max\{3\max_{|v|=1}g_v(t_0),3\delta\}, \quad  
	\int_0^{t_1} (k(t_1)-k(s))ds > ac.
\]
Such $t_1$ exists since 
$\lim_{t\to+\infty}\int_0^t (k(t)-k(s))ds = \delta \int_0^\infty (1-h(s))ds=+\infty$.
Since the integrands in the bound for $\xi_v(t)$ are nonnegative, we have 
for $t\ge \max\{t_0,t_1\}$ and $|v|=1$ that
\begin{equation}\label{eq:estimatexi}
	\xi_v(t) \ge \int_0^{t_0} \frac{g_v(t)-g_v(s)}{g_v(t)+\delta}ds +
	\int_0^{t_1} \frac{k(t)-k(s)}{g_v(t)+\delta}ds. 
\end{equation}
Assume that for some such $(t,v)$ we have that $g_v(t)+\delta \ge a$.
Since $a\ge 3\delta$, it follows that $g_v(t)\ge 2\delta$ and hence
\[
	\frac{g_v(t)}{g_v(t)+\delta} \ge \frac23.
\]
Furthermore, from $a\ge 3\max_{|v|=1}g_v(t_0)$ we get for $0\le s\le t_0$ that
\[  
	\frac{g_v(s)}{g_v(t)+\delta} \le \frac{g_v(t_0)}{a}\le \frac13.
\]
These two inequalities imply that the first integral in \eqref{eq:estimatexi} is bounded below
by $t_0/3\ge c$. If on the other hand $g_v(t)+\delta < a$ 
then the denominator of the second integral in \eqref{eq:estimatexi}
is at most $a$, so the integral is $\ge c$ by the choice of $t_1$.
This shows that $\xi_v(t)\ge c$ for all $|v|=1$ and $t\ge \max\{t_0,t_1\}$.
Since $c$ was arbitrary, condition  \eqref{eq:tangentline}
holds and hence $E_\psi$ has BCEH.
\end{proof}

The following observation will be used in the proof of Theorem \ref{th:main}.

\begin{proposition}\label{prop:exhaustion}
Denote by $\B$ the open unit ball in $\R^n$.
Let $E_\phi\subset \R^n$ be a closed convex set of the form
\eqref{eq:epi-phi} with $\Cscr^1$ boundary having BCEH, where the function
$\phi:\R^{n-1}\to\R$ is bounded from below and strictly convex near infinity.
Then there is an $r_0>0$ such that for every $r\ge r_0$ 
the convex hull $\Conv(E_\phi \cup r\cB)=\{y\ge \psi(x)\}$ is a closed convex set 
with BCEH, and $\psi:\R^{n-1}\to \R$ is a convex function of class 
$\Cscr^1$ such that $\psi\le\phi$ and these functions agree near infinity. 
Furthermore, if $r\ge r_0$ is large enough then the function $\phi_t:\R^{n-1}\to\R$ defined by
\begin{equation}\label{eq:phit}
	\phi_t(x)= (1-t)\phi(x) + t \psi(x),\quad x\in \R^{n-1} 
\end{equation}
is strictly convex for every $t\in (0,1)$, and for any $0<t_0 <t_1<1$ the closure of the set 
\[
	\{(x,y) \in\R^n:  \phi_{t_1}(x) <y< \phi_{t_0}(x) \} 
\]
is a strictly convex cap with the base in the strictly convex hypersurface $\{y=\phi_{t_0}(x)\}$.
\end{proposition}

\begin{proof}
Consider the function on $\R^{n-1}$ given by
\[
	\tilde \phi_r(x)=
	\left\{\begin{array}{lc}
	\min\{\phi(x),-\sqrt{r^2-|x|^2}\}, & |x|<r, \\
	\phi(x), &  \! |x|\ge r.
	\end{array}\right.
\]
(Note that $\tilde \phi_r$ may be discontinuous at the points of the sphere $|x|=r$.)
The convex hull of its epigraph $E_{\tilde \phi_r}$ equals $\Conv(E\cup r\cB)$, 
which is closed by Theorem \ref{th:GaleKlee} (iv), and the set 
$h(E,r\cB)=\Conv(E\cup r\cB)\setminus E$ is bounded since $E$ has BCEH. 
By smoothing $\tilde \phi_r$ we get a function $\tilde \psi_r$ 
of class $\Cscr^1$ which agrees with $\phi$ near infinity  
such that $\Conv(E_{\tilde\psi_r}) = \Conv(E\cup r\cB)$.
By \cite[Theorem 3.2]{GriewankRabier1990} we conclude that 
$\Conv(E\cup r\cB)$ has $\Cscr^1$ boundary, so it is the epigraph $E_{\psi_r}$  
of a convex function $\psi_r:\R^{n-1}\to\R$ of class $\Cscr^1$ which agrees
with $\phi$ near infinity. 

Since $\phi$ grows at least linearly, there is a function $\tau(r)$ defined for  
$r\in \R_+$ large enough such that $\psi_r(x)=-\sqrt{r^2-|x|^2}$ for $|x|\le \tau(r)$
and $\tau(r) \to +\infty$ as $r\to+\infty$. By choosing $r$ large enough,
the compact set of points where the function $\phi$ fails to be strictly convex is contained
in the ball $|x|<\tau(r)$. Since on this ball we have that 
$\psi_r(x)=-\sqrt{r^2-|x|^2}$ which is strictly convex, the convex combinations 
$\phi_t$ in \eqref{eq:phit} of $\phi$ and $\psi=\psi_r$ are strictly convex on $\R^{n-1}$
for all $0<t<1$. For such $r$, the last statement in the proposition is evident.
(Note that the strictly convex functions $\rho_t(x,y)=\exp(\psi_t(x)-y)-1$ 
for $t\in (0,1)$ correspond to those used in Section \ref{sec:pushing}.) 
\end{proof}

%
%
\section{Proof of Theorem \ref{th:main}}\label{sec:proof}

For the definition and the main theorem on Oka manifolds, 
see \cite[Definition 5.4.1 and Theorem 5.4.4]{Forstneric2017E}. 
We shall use the following version of the Oka principle; 
see \cite[Theorem 1.3]{Forstneric2022Oka}.

%
%
\begin{theorem}\label{th:Oka}
Assume that $X$ is a Stein manifold, $K$ is a compact $\Oscr(X)$-convex set in $X$, 
$X'$ is a closed complex subvariety of $X$, 
$\Omega$ is an Oka domain in a complex manifold $Y$, $f:X\to Y$ is a 
continuous map which is holomorphic on a neighbourhood of $K$, 
$f|_{X'}:X'\to Y$ is holomorphic, and $f(X\setminus \mathring K) \subset \Omega$. 
Then there is a homotopy $\{f_t\}_{t\in[0,1]}$ of continuous maps $f_t:X\to Y$ 
connecting $f=f_0$ to a holomorphic map $f_1:X\to Y$ such that for every $t\in [0,1]$ 
the map $f_t$ is holomorphic on a neighbourhood of $K$, it agrees with $f$ on $X'$,
it approximates $f$ uniformly on $K$ and uniformly in $t\in[0,1]$ as closely as desired, 
and $f_t(X\setminus \mathring K) \subset \Omega$.
\end{theorem}

%
%
\begin{proof}[Proof of Theorem \ref{th:main}]
By Proposition \ref{prop:epigraph} there are complex coordinates 
$z=(z',z_n)$ on $\C^n$ such that the given set $E$ is an epigraph of the form \eqref{eq:EinCn}. 
We shall write $z=(x,y)$ where  $x=(z',\Re z_n)\in \C^{n-1}\times \R\cong \R^{2n-1}$ and 
$y=\Im z_n\in\R$, so $E=E_\phi=\{y\ge \phi(x)\}$ where $\phi\ge 0$ is a convex function as in 
Proposition \ref{prop:epigraph}. Let the set $K\subset X$ and the map $f_0:K\to \C^n$ 
be as in the theorem; in particular, $f_0(bK)\subset \C^n \setminus E$.
Thus, there are an open neighbourhood $U\subset X$ of $K$ and $\epsilon>0$
such that $f_0$ is holomorphic in $U$ and 
$f_0(U\setminus \mathring K)\subset \C^n\setminus E_{\phi-\epsilon}$.
By Azagra \cite[Theorem 1.8]{Azagra2013} there is a 
a real analytic strictly convex function $\phi_0:\R^{2n-1}\to \R$ such that 
$\phi-\epsilon < \phi_0 <\phi$. Its epigraph $E_0=\{(x,y) \in\C^n : y\ge \phi_0(x)\}$
is a closed strictly convex set with real analytic boundary which has BCEH
by Lemma \ref{lem:approximation}, and 
$f_0(U\setminus \mathring K)\subset \C^n\setminus E_{0}$.

Let $\B$ denote the open unit ball in $\C^n$ centred at $0$. Recall the notation $h(E,K)$
in \eqref{eq:h}. Pick a  number $r_0>0$. We can find an increasing sequence 
$r_k>0$ diverging to infinity such that 
\begin{equation}\label{eq:rk}
	\overline{h(E_0,r_k \cB)} \subset r_{k+1}\B\quad \text{for}\ k=0,1,2,\ldots.
\end{equation}
Indeed, since $E_0$ has BCEH, the set $h(E_0,r_k \cB)$ is bounded for each $k$, 
and hence \eqref{eq:rk} holds if the number $r_{k+1}$ is chosen large enough. Set 
\[
	E_{k+1}=\Conv(E_0\cup r_k\cB) = E_0\cup h(E_0,r_k\cB) 
	\quad \text{for $k=0,1,2,\ldots$}. 
\]
We clearly have that
$E_0\subset E_1\subset\cdots \subset \bigcup_{k=0}^\infty E_k =\C^n$. 
Furthermore, \eqref{eq:rk} shows that for $j=0,1,\ldots,k+1$ we have that 
$E_0\subset E_{j}\subset E_0\cup r_{k+1}\cB$ and hence
\begin{equation}\label{eq:Ekplustwo}
	E_{k+2} =	\Conv(E_j \cup r_{k+1}\cB) 
	\ \ \text{for}\ j=0,1,\ldots,k+1.
\end{equation}
Proposition \ref{prop:exhaustion} shows that for each $k=1,2,\ldots$ 
we have $E_k=\{y\ge \phi_k(x)\}$ where $\phi_k$ a convex function 
of class $\Cscr^1$ which agrees with $\phi_0$ near infinity, and $E_k$ has BCEH.
Hence,  
\[
	\Omega_k=\C^n\setminus E_k=\{(x,y)\in\C^n : y<\phi_k(x)\}
\]
is an Oka domain for every $k=0,1,\ldots$ by Corollary \ref{cor:Oka}. 
In view of $E_{k+2} = \Conv(E_k \cup r_{k+1}\cB)$ (see \eqref{eq:Ekplustwo}),
Proposition \ref{prop:exhaustion} also shows that if $r_{k+1}$ is chosen 
large enough then the function 
\begin{equation}\label{eq:psit}
	\psi_t=(1-t)\phi_k + t \phi_{k+2}:\C^{n-1}\times\R\to\R
\end{equation}
is strictly convex for every $t\in (0,1)$, and for each $0<t_0 < t_1<1$ the closure of the set 
\begin{equation}\label{eq:capk}
	C = \{(x,y): \psi_{t_1} < y < \psi_{t_0}\} 
\end{equation}
is a strictly convex cap as described in Section \ref{sec:pushing}.
(Note that the strictly convex functions $\rho_t(x,y)=\exp(\psi_t(x)-y)-1$ 
for $t\in (0,1)$ correspond to those used in Section \ref{sec:pushing}.) 
 
Choose an exhaustion $D_0\subset D_1\subset\cdots \subset \bigcup_{k=0}^\infty D_k=X$ 
by smoothly bounded, relatively compact, strongly pseudoconvex domains with 
$\Oscr(X)$-convex closures such that $K\subset D_0 \subset \bar D_0 \subset U$. 
For consistency of notation we set $D_{-1}=K$.
We now construct a sequence of holomorphic maps $f_k:\bar D_k\to \C^n$ 
satisfying the following conditions for $k=0,1,2,\ldots$:
\begin{enumerate}[\rm (a)]
\item $f_k(\overline {D_k \setminus D_{k-1}}) \subset \Omega_k=\C^n\setminus E_k$, 
\item $f_{k+1}(\overline {D_k \setminus D_{k-1}}) \subset \Omega_k$, and 
\item $f_{k+1}$ approximates $f_k$ uniformly on $\bar D_{k-1}$ as closely as desired.
\end{enumerate}
For $k=0$ the initial map $f_0$ in Theorem \ref{th:main} satisfies condition (a) 
while conditions (b) and (c) are void. Assuming inductively that we found maps 
$f_0,\ldots, f_k$ satisfying these conditions, the construction of the next map
$f_{k+1}$ is made in two steps as follows. 

By compactness of the set $f_k(bD_k)\subset \Omega_k=\{y<\phi_k(x)\}$ 
we can choose $t_0\in (0,1)$ small enough such that 
$f(bD_k)\subset  \{y<\psi_{t_0}(x)\}$, where the function $\psi_{t}$ $(t\in[0,1])$
is given by \eqref{eq:psit}. By \eqref{eq:rk} we can also choose 
$t_1\in (t_0,1)$ sufficiently close to $1$ such that 
\[
	E_{k+1} \subset \{(x,y): y\ge \psi_{t_1}(x)\}.
\]
Proposition \ref{prop:pushing} applied to the map $f_k:\bar D_k\to\C^n$,
the set $E_k$, and the strictly convex cap $C$ \eqref{eq:capk} 
(which corresponds to $C_1$ in Proposition \ref{prop:pushing})
gives holomorphic map $g_k:\bar D_k\to \C^n$ approximating $f_k$ on 
$\overline{D}_{k-1}$ and satisfying 
\begin{equation}\label{eq:gk}
	g_k(bD_k)\subset \{(x,y): y < \psi_{t_1}(x)\} \subset 
	\C^n\setminus E_{k+1} = \Omega_{k+1} 
	\ \ \ \text{and}\ \ \   
	g_k(\overline{D_k\setminus D_{k-1}})  \subset \Omega_k.
\end{equation}

In the second step, we use that $\Omega_{k+1}$ is an Oka domain. 
Since $\Omega_{k+1}$ is contractible and $g_k(bD_k)\subset \Omega_{k+1}$ 
by \eqref{eq:gk}, $g_k$ extends from $\bar D_k$ to a continuous map $X\to \C^n$ sending  
$X\setminus D_k$ to $\Omega_{k+1}$. Theorem \ref{th:Oka} applied to $g_k$ gives a 
holomorphic map $f_{k+1}:\bar D_{k+1}\to\C^n$ approximating $g_k$  
on $\bar D_k$ and satisfying 
$f_{k+1}(\overline{D_{k+1} \setminus D_{k}}) \subset \Omega_{k+1}$ 
(which is condition (a) for $k+1$) and 
$f_{k+1}(\overline{D_k\setminus D_{k-1}})  \subset \Omega_k$ (condition (b)).
Since $f_{k+1}$ approximates $g_k$ on $\bar D_k$ and
$g_k$ approximates $f_k$ on $\bar D_{k-1}$, $f_{k+1}$ also satisfies 
condition (c). This completes the induction step.

If the approximations are close enough then the sequence $f_k$ 
converges uniformly on compacts in $X$ to a holomorphic $f:X\to\C^n$. 
Conditions (a)--(c) and the fact that the sets $E_k$ exhaust $\C^n$
imply that $f$ is a proper holomorphic map satisfying 
$f(X\setminus \mathring K)\subset \Omega_0=\C^n\setminus E_0$. 
To construct proper holomorphic immersions and embeddings in suitable dimensions 
given in the theorem, we use the general position argument at every step to ensure that 
every map $f_k$ in the sequence is an immersion or an embedding. 
(See e.g.\ \cite[Corollary 8.9.3]{Forstneric2017E}.) If the convergence is fast enough 
then the same holds for the limit map $f$ by a standard argument.
\end{proof}

%
%
\begin{proof}[Proof of Corollary \ref{cor:main}]
Given a holomorphic map $f_0:K\to\C^n$ with $f_0(bK)\subset \C^n\setminus E_\phi$
as in Theorem \ref{th:main}, Proposition \ref{prop:approximation} furnishes
a closed convex set $E_\psi\supset E_\phi$ with BCEH such that 
$f_0(bK) \subset \C^n\setminus E_\psi$. Applying Theorem \ref{th:main} with
$E_\psi$ gives the desired conclusion.
\end{proof}

We have the following analogue of Theorem \ref{th:main}
with interpolation on a closed complex subvariety of $X$. 
Unlike in the above corollary, approximation of $E$ from the outside
by convex sets enjoying BCEH cannot be used since the subvariety
$f(X')$ may have zero distance to $bE$.
This results extends the case of \cite[Theorem 15]{ForstnericRitter2014}  
when $E$ is a compact convex set.

%
%
\begin{theorem}\label{th:interpolation}
Let $E$ be a closed convex set in $\C^n$ $(n>1)$ with  $\Cscr^1$ boundary
which is strictly convex near infinity and has bounded convex exhaustion hulls. 
Let $X$ be a Stein manifold, $K \subset X$ be a compact $\Oscr(X)$-convex set, 
$U\subset X$ be an open set containing $K$, $X'$ be a closed complex subvariety of $X$, 
and $f_0 : U\cup X'\to\C^n$ be a holomorphic map such that $f_0|_{X'}: X'\to\C^n$ 
is proper holomorphic and $f_0(bK \cup (X'\setminus K)) \cap E =\varnothing$. 
Given $\epsilon >0$ there exists a proper holomorphic map 
$f: X\to \C^n$ satisfying the following conditions:
\[
	(a)\ f(X\setminus \mathring K) \subset \C^n\setminus E, \qquad (b)\ \|f-f_0\|_K<\epsilon, 
	\qquad (c)\ f|_{X'}=f_0|_{X'}.
\] 
If $2\dim X \le n$ then $f$ can be chosen an immersion (and an embedding if $2\dim X+1\le n$) 
provided that $f_0|_{X'}$ is one. 
\end{theorem}

\begin{proof}
This is proved by a small modification of the proof of Theorem \ref{th:main},
similar to the one in \cite[proof of Theorem 15]{ForstnericRitter2014}.
The initial step in the proof, approximating $E$ from the outside by a strictly convex set,
is  unnecessary since $bE$ is strictly convex near infinity. The main (and essentially the only)
change comes in the choice of the exhaustion $D_k$ of the Stein manifold $X$. 
In the inductive step when constructing the map $f_{k+1}$, we must assume in addition 
that $f_k(bD_k\cap X') \subset \Omega_{k+1}=\C^n\setminus E_{k+1}$. 
Then, we push the image of $bD_k$ out of $E_{k+1}$ by the same method 
as before, using Proposition \ref{prop:pushing} but ensuring that the modifications 
are kept fixed on $X'$ and small near $bD_k\cap X'$. 
This is possible since the method from \cite{DrinovecForstneric2010AJM} 
is applied locally near $bD_k$ (away from $bD_k\cap X'$),
and these local modifications are glued together by preserving the value
of the map on $X'$. We refer to \cite[proof of Theorem 15]{ForstnericRitter2014}
for a more precise description. This gives the next holomorphic map
$f_{k+1}:X\to \C^n$ satisfying $f_{k+1}(X\setminus D_k)\subset \Omega_{k+1}$,
$f_{k+1}|_{X'}= f_k|_{X'}$, and conditions (b) and (c) in the proof of Theorem \ref{th:main}. 
We then choose the next domain $D_{k+1}\subset X$ big enough such that 
$f_{k+1}(bD_{k+1}\cap X') \subset \Omega_{k+2}=\C^n\setminus E_{k+2}$.
This is possible since the map $f_{k+1}|_{X'}=f_0|_{X'}:X'\to \C^n$ is proper,
$f_0(X' \setminus \mathring K)\subset \Omega=\C^n\setminus E$, and the domain
$\Omega_{k+2}$ agrees with $\Omega$ near infinity by the construction.
Clearly the induction step is now complete. Assuming that the approximations
are close enough, the sequence $f_k$ converges to a limit holomorphic 
map $f:X\to\C^n$ satisfying the stated conditions.
\end{proof}

%
%
%
%
\subsection*{Acknowledgements}
The first named author is supported by grants P1-0291, J1-3005, and N1-0137 from ARRS, 
Republic of Slovenia. The second named author is supported by the 
European Union (ERC Advanced grant HPDR, 101053085) and grants P1-0291, J1-3005, and N1-0237 from ARRS, Republic of Slovenia. The authors wish to thank Antonio Alarc\'on for 
helpful discussions and information concerning the case pertaining to minimal surfaces.


%




\vspace*{5mm}
\noindent Barbara Drinovec Drnov\v sek

\noindent Faculty of Mathematics and Physics, University of Ljubljana, Jadranska 19, SI--1000 Ljubljana, Slovenia

\noindent 
Institute of Mathematics, Physics and Mechanics, Jadranska 19, SI--1000 Ljubljana, Slovenia.

\noindent 
e-mail: {\tt barbara.drinovec@fmf.uni-lj.si}

\vspace*{5mm}
\noindent Franc Forstneri\v c

\noindent Faculty of Mathematics and Physics, University of Ljubljana, Jadranska 19, SI--1000 Ljubljana, Slovenia

\noindent 
Institute of Mathematics, Physics and Mechanics, Jadranska 19, SI--1000 Ljubljana, Slovenia

\noindent e-mail: {\tt franc.forstneric@fmf.uni-lj.si}

\end{document}